\documentclass[11pt,reqno]{amsart}

\usepackage{a4wide}
  \usepackage{amssymb}
  \usepackage{amsmath}
  \usepackage{amsthm}
  \usepackage{amsfonts}
  \usepackage{amscd}
  \usepackage{xspace}
  \usepackage[foot]{amsaddr}
  \usepackage{boxedminipage}
  \usepackage{nomencl}

  \usepackage{booktabs}
  \usepackage{multirow}
  \usepackage[figuresright]{rotating}
  \usepackage[format=hang,font=small]{caption} 
  \usepackage[curve,graph]{xypic}
  \usepackage{lscape}
  \usepackage{subfigure}

  \usepackage{stmaryrd}

  \makeatletter
  \CheckCommand*\refstepcounter[1]{\stepcounter{#1}%
      \protected@edef\@currentlabel
       {\csname p@#1\endcsname\csname the#1\endcsname}%
  }
  \renewcommand*\refstepcounter[1]{\stepcounter{#1}%
    \protected@edef\@currentlabel
      {\csname p@#1\expandafter\endcsname\csname the#1\endcsname}%
  }
  \def\labelformat#1{\expandafter\def\csname p@#1\endcsname##1}
  \DeclareRobustCommand\Ref[1]{\protected@edef\@tempa{\ref{#1}}%
     \expandafter\MakeUppercase\@tempa
  }
  \makeatother

  \makeatletter
  \newcommand{\numberlike}[2]{%
     \expandafter\def\csname c@#1\endcsname{%
         \expandafter\csname c@#2\endcsname}%
  }
  \makeatother

  \def\DefaultNumberTheoremWithin{section}

  \theoremstyle{plain}
  \newtheorem{lem}{Lemma}
     \numberwithin{lem}{\DefaultNumberTheoremWithin}
     \labelformat{lem}{Lemma~#1}
  
     \numberwithin{claim}{\DefaultNumberTheoremWithin}
     \numberlike{claim}{lem}
     \labelformat{claim}{Claim~#1}
  \newtheorem{thm}{Theorem}
     \numberwithin{thm}{\DefaultNumberTheoremWithin}
     \numberlike{thm}{lem}
     \labelformat{thm}{Theorem~#1}
  \newtheorem{cor}{Corollary}
     \numberwithin{cor}{\DefaultNumberTheoremWithin}
     \numberlike{cor}{lem}
     \labelformat{cor}{Corollary~#1}
  \newtheorem{prop}{Proposition}
     \numberwithin{prop}{\DefaultNumberTheoremWithin}
     \numberlike{prop}{lem}
     \labelformat{prop}{Proposition~#1}
  \newtheorem{conj}{Conjecture}
     \numberwithin{conj}{\DefaultNumberTheoremWithin}
     \numberlike{conj}{lem}
     \labelformat{conj}{Conjecture~#1}

  \theoremstyle{definition}
  \newtheorem{defn}{Definition}
     \numberwithin{defn}{\DefaultNumberTheoremWithin}
     \numberlike{defn}{lem}
     \labelformat{defn}{Definition~#1}

  \theoremstyle{definition}
  
     \numberwithin{question}{\DefaultNumberTheoremWithin}
     \numberlike{question}{lem}
     \labelformat{question}{Question~#1}

  \theoremstyle{definition}
  
     \numberwithin{problem}{\DefaultNumberTheoremWithin}
     \numberlike{problem}{lem}
     \labelformat{problem}{Problem~#1}

  \newtheorem{rem}{Remark}
     \numberwithin{rem}{\DefaultNumberTheoremWithin}
     \numberlike{rem}{lem}
     \labelformat{rem}{Remark~#1}

     \numberwithin{example}{\DefaultNumberTheoremWithin}
     \numberlike{example}{lem}
     \labelformat{example}{Example~#1}

     \labelformat{case}{Case~#1}
     \numberwithin{case}{lem}

     \labelformat{step}{Step~#1}
     \numberwithin{step}{lem}

  \labelformat{equation}{(#1)}
  \labelformat{figure}{Figure~#1}
  \labelformat{chapter}{Chapter~#1}
  \labelformat{appendix}{Appendix~#1}
  \labelformat{section}{\S#1}
  \labelformat{subsection}{\S#1}
  \labelformat{subsubsection}{\S#1}

  \def\eqref{\ref}

\def\cocoa{{\hbox{\rm C\kern-.13em o\kern-.07em C\kern-.13em o\kern-.15em A}}}

\def\rank{{\rm rank}}
\def\height{{\rm height}}

\def\deg{{\rm deg}}
\def\pmd{{\mathrm{pmd}}}
\def\iin{{\rm in}}

\def\w{{\mathfrak w}\,}

\def\implies{\ifmmode\Rightarrow \else
        \unskip${}\Rightarrow{}$\ignorespaces\fi}

\def\K{{\mathbb{K}}}
\def\R{{\mathbb{R}}}

\def\RR{\mathbb{R}}

\def\Sym{\mathrm{Sym}}
\def\V{\mathcal{V}}

\begin{document}
	
\begin{abstract}
Let $\K $ be a field, $[n]=\{1,\ldots,n\} $ and $H=([n],E)$ be a hypergraph.
For an integer $ d \geq 1 $ the Lov\'{a}sz-Saks-Schrijver ideal (LSS-ideal) $ L_H^{\K} (d) 
	\subseteq \K\big[\,y_{ij}~:~(i,j) \in [n] \times [d]\, \big]$ 
is the ideal generated by the polynomials 
$ f^{(d)}_{e}= \sum\limits_{j=1}^{d} \prod\limits_{i \in e} y_{ij}$ for edges $e$ of $ H $.

In this paper for an algebraically closed field $ \K $ and a 
$k$-uniform hypergraph $H=([n],E)$ we employ a connection between LSS-ideals and coordinate sections of the closure of 
	the set $S_{n,k}^d$ of homogeneous degree $k$ symmetric tensors in $n$ variables of $\text{rank} \leq d$ 
to derive results on the irreducibility of its coordinate sections.
To this end we provide results on primality and the complete intersection property of $ L_H^{\K} (d) $. We then use the 
combinatorial concept of positive matching decomposition of a hypergraph  $ H $ to provide bounds on when $L_H^{\K}(d)$ turns
prime to provide results on the irreducibility of coordinate sections of $ S_{n,k}^d $.
	
\end{abstract}	

\address{Shekoofeh Gharakhloo, Department of Mathematics, Institute for Advanced Studies in Basic Sciences (IASBS), Zanjan 45137-66731, Iran}
\email{gh.shekoofeh@iasbs.ac.ir, gh.shekoofeh@gmail.com}

\address{Volkmar Welker, Fachbereich Mathematik und Informatik, Philipps-Universit\"at Marburg, 35032 marburg, Germany}
\email{welker@mathematik.uni-marburg.de}

\thanks{The authors are grateful to Rashid Zaare-Nahandi for numerous suggestions and discussions. The first author acknowledges support from Ministry of Science, Research and Technology of Iran for her research visit to Marburg, where most of the work on the project was done.} 
	
\title{Hypergraph LSS-ideals and coordinate sections of symmetric tensors}
\maketitle	
	
\section{Introduction}\label{section0}
  In this paper we use Lov\'asz-Saks-Schrijver-ideals 
  associated to hypergraphs in order to study coordinate sections of 
  the Zariski closure of the set of symmetric tensors with bounded rank. 

  Let $\K $ be a field, $[n]=\{1,\ldots,n\} $ and $H=([n],E)$ be a 
  hypergraph such that $E$ is a clutter; i.e. a set of sets that are pairwise incomparable with respect to 
  inclusion. For an integer $ d \geq 1 $ and 
  $ e \in E $ we consider the polynomials 
  \begin{align*}
    f^{(d)}_{e}= \sum\limits_{j=1}^{d} \prod\limits_{i \in e} y_{ij} 
  \end{align*}
  in the polynomial ring $ S=\K[y_{ij}:~i\in [n],~ j\in [d]] $. The ideal 
  \begin{align*}
    L^{\K}_{H}(d)=(f^{(d)} _e:~e\in E) \subseteq S 
  \end{align*}  
  is called the Lov\'{a}sz-Saks-Schrijver ideal (LSS-ideal) of $ H $. 
  LSS-ideals for the case when $H$ is a simple graph were first studied 
  in \cite{HMSW} and 
  later in \cite{ac-vw} for graphs and hypergraphs. These papers took their 
  motivation from the fact that when $\K = \R$ the vanishing 
  set of the ideal $L^{\K}_{H}(d) $ 
  coincides with the variety of orthogonal representations of the graph 
  complementary to $ H $. 
  Orthogonal representations of graphs were introduced by Lov\'{a}sz in 1979 
  \cite{Lo}. The variety of orthogonal representations was studied in work of 
  Lov\'asz, Saks and Schrijver \cite{LSS,LSS-correction}. 
  We refer to the book \cite{L} for a comprehensive treatment of orthogonal 
  representations, their varieties and their relations to important 
  combinatorial properties of graphs.

  We now  explain the connection of LSS-ideals of $ k $-uniform hypergraphs 
  and coordinate sections of the variety which is the closure of the set 
  of symmetric tensors of bounded rank.
  Recall that for an integer $ k\geq 1 $, a hypergraph $ H=(V,E) $ is called 
  $ k $-uniform if every edge $ e \in E $ has cardinality $ k $.
 
  Let $\K$ be an algebraically closed field.
  Consider the map 
  $ \phi :(\K^n)^d \longrightarrow \underbrace {\K^n \otimes\ldots \otimes\K^n}\limits_{k}  $ 
  which sends $ (v_1,\ldots,v_d) \in (\K^n)^d $ to the symmetric tensor 
  $$\sum_{j=1}^d \underbrace{v_j\otimes \cdots \otimes v_j}_{k} =  
     \sum_{j=1}^d \sum_{1 \leq i_1, \ldots,i_k \leq n} 
     (v_j)_{i_1} \cdots (v_j)_{i_k}\, e_{i_1}\otimes\ldots\otimes e_{i_k} 
     \in \underbrace {\K^n \otimes\ldots \otimes\K^n}_k $$ 
  where $ \{e_1,\ldots,e_n\} $ is the standard basis of $ \K^n $.
  The Zariski closure of the image of $\phi$ is the variety $S_{n,k}^d$ of 
  symmetric tensors of (symmetric) rank $\leq d$. 
  For background on tensor rank and the geometry of varieties of bounded rank 
  tensors we refer the reader to \cite{Lans}. 
  Let $ \V(L_H^{\K}(d)) $ be the vanishing locus of the ideal $ L_H^{\K}(d) $. 
  Note that for $1 \leq i_i < \cdots < i_k \leq n$ the coefficient of 
  $e_{i_1} \otimes \cdots \otimes e_{i_k}$ in 
  $\phi(v_1,\ldots, v_d)$ is 
  $$\sum_{j=1}^d (v_j)_{i_1} \cdots (v_j)_{i_k} = 
    f^{(d)}_{\{i_1 < \cdots < i_k\}} (v_1,\ldots, v_d).$$
  Thus if we restrict the map $ \varphi $ to $\V(L_H^{\K}(d))$, then we obtain 
  a parameterization of the coordinate section of $S_{n,k}^d$ with $0$ coefficient at  
  $e_{i_1} \otimes \cdots\otimes e_{i_k}$ for $\{i_1,\ldots, i_k\} \in E$. 
  In particular, the Zariski-closure of the image of the restriction is 
  irreducible if $ L_H^{\K}(d) $ is prime. 
  Therefore, proving primality for $ L_H^{\K}(d) $ is a tool for 
  deriving the irreducibility of coordinate sections of $S_{n,k}^d$. 

  In order to approach primality of $ L_H^{\K}(d) $ we prove in 
  Section \ref{sec:2} the following 
  theorems, which hold for general fields $\K$.

  \begin{thm}\label{th1}
    Let $H=([n],E)$ be a $k$-uniform hypergraph and $d \geq 2$. If $L_H^{\K}(d)$ 
    is prime, then $L_H^{\K}(d)$ is a complete intersection.
  \end{thm}

  \begin{thm}\label{th2}
    Let $H=([n],E)$ be a $k$-uniform hypergraph and $d \geq 2$. If 
    $L_H^{\K}(d-1)$ is a complete intersection, then 
    $L_H^{\K}(d)$ is prime.
  \end{thm} 

  These theorems generalize results in the graph case from \cite{ac-vw}. 
  In order to understand when the ideal $ L_H^{\K}(d) $ is complete 
  intersection, we employ the following combinatorial tool from 
  \cite[Definition 5.1 and 5.3]{ac-vw}.
  Let $H = (V,E)$ be a hypergraph and recall that a subset $M \subseteq E$ is
  a \textit{matching} if $e \cap e' = \varnothing$ for $e,e' \in M$ and $e \neq e'$.

  \begin{defn}
    \begin{itemize}
      \item[(i)]
        A \textit{positive matching} of $H$ is a matching $M$ of $H$ such that 
        there exists a weight function $\w \colon V \to \mathbb{R}$ satisfying:
        \begin{equation}
          \begin{split}
		  \w(e) := & \sum\limits_{i \in e} \w(i) >0, \quad\text{ if } e \in M, \\
		  \w(e) := & \sum\limits_{i \in e} \w(i) <0, \quad\text{ if } e \notin M.
          \end{split}
        \end{equation}
      \item[(ii)]
        A \textit{positive matching decomposition} (or \textit{pmd}) of $H$ 
        is a partition $E = \bigcup_{i=1}^p E_i$ of $E$ such that $E_i$ is a 
	positive matching of $(V, E \setminus \bigcup_{j=1}^{i-1} E_j)$, 
        for $i = 1, \ldots, p$. The $E_i$ are called the \textit{parts} of the 
        pmd. The smallest $p$ for which $ H $ admits a pmd with $ p $ parts 
	will be denoted by $ \pmd(H) $.
    \end{itemize}
  \end{defn}

  In \cite[Proposition 2.4, Lemma 5.5]{ac-vw} it is proved that for a 
  hypergraph $H = (V,E)$:

  \begin{align}
    \label{eq:imp}
    \pmd(H) \leq d & \Rightarrow L^{\K}_{H}(d)  \text{ is a complete intersection}.
  \end{align}

  Thus for proving irreducibility of the coordinate section of $S_{n,k}^d$ 
  corresponding to $H$ we can proceed as follows. First we show that 
  $d-1 \geq \pmd(H)$. 
  Then by \eqref{eq:imp} the ideal $L^{\K}_{H}(d-1)$ is a complete 
  intersection and by \ref{th2} it follows that $L^{\K}_{H}(d) $ is prime. 
  As a consequence, the coordinate section of $S_{n,k}^d$ corresponding to $H$ 
  is irreducible. 

  Our next goal in Section \ref{sec:3} is to deduce bounds for \textit{pmd} in 
  order to apply the above mentioned method. We compute $\pmd(H)$  
  for various families of hypergraphs.
  For example we prove the following result about hypergraphs which are trees.
  Note that our definition of a $k$-uniform tree is more 
  restrictive than other definitions from the literature
  (see \cite{b}, \cite{bdcv}, \cite{np}).
  For a hypergraph $H$ we denote by $\Delta(H)$ the maximal degree of a 
  vertex.

  \begin{thm}\label{TREE}
    Let $ H=(V,E) $ be a $ k $-uniform tree. Then 
    $\pmd(H)=\Delta(H)$.
  \end{thm}

  To the best of our knowledge the strongest general bound known for the symmetric rank 
  of tensors of order $ k $ on $ \K^n $ is $ \binom{n+k-1}{k}-n+1$ 
  (see \cite[Corollary 5.2]{LANS}). Therefore, in order to make a sensible 
  contribution we need to take $ d $ below this bound.

  \begin{cor}
    Let $ H=([n],E) $ be a $k$-uniform tree. Then the coordinate sections of 
    the variety $S_{n,k}^d$ with respect to $ H $ for 
    $ \Delta(H)+1 \leq d \leq \binom{n+k-1}{k}-n $, are irreducible.
  \end{cor} 

  At the end of Section \ref{sec:3}, we state a conjecture which says that
  when $ H $ is a $3$-uniform hypergraph then $\pmd(H)$ is bounded from 
  above by a 
  quadratic function in the number of vertices. Note, that the bound from 
  \cite{LANS} on the (symmetric) tensor rank is cubical in the number of
  vertices.
 
\section{Algebraic properties of LSS--ideals for hypergraphs}
\label{sec:2}

  In this section we prove \ref{th1} and \ref{th2}. 
  The proof uses an extension and modifications of the approach 
  from \cite{ac-vw} in the graph case and some new proof methods. 
  Indeed, some of the facts used in \cite{ac-vw} do not generalize to 
  hypergraphs and have to replaced by alternative arguments.
  As in \cite{ac-vw} we heavily rely on facts about symmetric algebras from 
  \cite{lla} and \cite{CH}.

  Let $R$ be a ring and $M$ be an $R$-module given as the cokernel of an 
  $R$-linear map $f : R^m \xrightarrow{A^T} R^n$ represented by a matrix 
  $A \in R^{m \times n}$. 
  In this situation the symmetric algebra 
  $\Sym_R(M)$ of $M$ is isomorphic to  the quotient of $\Sym_R(R^n) 
  = R[x_1,\ldots, x_n]$ by the ideal generated by the entries of the vector
  $A^T (x_1,\ldots, x_m)^T$. 
  For a matrix $B \in R^{m \times n}$ and a number $t \geq 1$ we denote by
  $I_t(B)$ the ideal generated by the $(t \times t)$-minors of $B$. 

  The key property of symmetric algebras we use in this section is stated in 
  the following proposition from \cite{ac-vw} where the relevant results from
  \cite{lla} and \cite{CH} are summarized.

  \begin{prop}[{\cite[Theorem 4.1]{ac-vw}}]
    \label{Sym}
    Let $R$ be a complete intersection and $M$ be an $R$-module 
    presented by the matrix $A \in R^{m \times n}$. Then 
    \begin{itemize}
      \item[(1)] $\Sym_R(M) $ is a complete intersection $\Leftrightarrow$ 
	      $\height(I_t(A)) \geq  m - t + 1 $ for all $ t=1,\ldots,m $.
      \item[(2)] $\Sym_R(M) $ is a domain and $ I_m(A)\neq 0 $ 
	      $\Leftrightarrow$  $ R $ is a domain and 
              $\height (I_t(A)) \geq  m - t + 2 $ for all $ t=1,\ldots,m $. 
    \end{itemize}
  \end{prop}

  We fix some notation that we will use throughout this section without further
  reference. For a $k$-uniform hypergraph $H = ([n],E)$ and a number 
  $d \geq 1$ we write 

  \bigskip

  {\small
  \begin{tabular}{cc} $S= \K\big[y_{ij} : i \in [n], j\in[d]\big]$, &  
     $S'=\K\big[y_{ij}: i \in [n-1], j\in[d]\big]$, \\ 
	  & \\
     $H'=H\setminus \{n\}$, & $R = S'/L_{H'}^\K(d)$, \\
	  & \\
	  $U=\Big\{\{i_1,\ldots,i_{k-1}\} \subseteq [n-1]\, 
	  \Big|\, \{i_1,\ldots,i_{k-1},n\} \in E\Big\}$, &   
     $u=|U|$. 
  \end{tabular}
  }

  \bigskip

  First we need to understand that for a $k$-uniform hypergraph $H$ the ideal  
  $L_H^{\K}(d) \subseteq S$ can be seen as the defining ideal of a symmetric 
  algebra.

  \begin{rem}
    \label{rem:symmetric}
    Let $H=([n],E)$ be a $k$-uniform hypergraph. 
    Then $ S/L_H^{\K}(d)$ is the 
    symmetric algebra of the cokernel of the linear map 
    $R^u \xrightarrow{A^T} R^d$ defined by the $u\times d$ matrix 
    $A$ where $$A=\Big(y_{i_1 j} y_{i_2 j} \cdots 
	 y_{i_{k-1}j}\Big)_{\{i_1,\ldots, i_{k-1} \} 
	 \in U, j \in [d]} \in S'^{\,u \times d}.$$
  \end{rem}

  In order to apply \ref{Sym} to the situation described in 
  \ref{rem:symmetric} we need to prove the two following lemmas.

  \begin{lem}
    \label{lem:nonvan}
    Let $H = ([n],E)$ be a $k$-uniform hypergraph.
    Then for every $2 \leq t \leq u$ the $t$-minor $f_t$ of 
    $$A=(y_{i_1 j} y_{i_2 j} \cdots y_{i_{k-1}j})_{\{i_1,\ldots, i_{k-1} \} 
	  \in U, j \in [d]} \in S^{\,u \times d}$$
    corresponding to the first $t$ rows and columns is non-zero in 
    $S/L_H^{\K}(d)$.
  \end{lem}
  \begin{proof}
    First we show. 

    \noindent {\sf Claim:} $f_t$ is non-zero in $S$.

    \noindent $\triangleleft$ 
    We prove the assertion by induction on $t$. For $t=2$ we have 
    $$f_2=(y_{i_11} \cdots y_{i_{k-1}1}) (y_{j_12} \cdots y_{j_{k-1}2}) -
	   (y_{j_11} \cdots y_{j_{k-1}1}) (y_{i_12} \cdots y_{i_{k-1}2})$$
    for two different $(k-1)$-sets 
    $\{ i_1 < \cdots < i_{k-1}\} , \{ j_1 < \cdots < j_{k-1}\} \in U$.
    Clearly, $f_2 \neq 0$. 

    Assume that for $2 \leq t \leq u-1$ the claim is true and
    $0 \neq f_t\in I_t(A)$. For $1 \leq j \leq u$, let 
    $\{i_1^{(j)},\ldots, i_{k-1}^{(j)}\}$ be 
    the $(k-1)$-subset in $U$ indexing the $j$\textsuperscript{th} row of $A$. 
    We expand the minor $f_{t+1}$ along the 
    $(t+1)$\textsuperscript{st} column and we get 
    \begin{align*}
      f_{t+1}=\Big(y_{i_1^{(t+1)}\,t+1} \cdots y_{i_{k-1}^{(t+1)}\,t+1}\Big)\, f_t+
      \sum_{j=1}^{t}(-1)^{j+t+1}\Big(y_{i_1^{(j)}\,t+1} \cdots 
        y_{i_{k-1}^{(j)}\,t+1}\Big)\,g_j
    \end{align*}
    where for every $1\leq j \leq t$, $g_j$ is a $t$-minor corresponding to 
    the first $t$ columns of $A$ and the rows with index in $[t+1] 
    \setminus \{j\}$. 
    Since monomials divisible by 
    $y_{i_1^{(t+1)}\,t+1} \cdots y_{i_{k-1}^{(t+1)}\,t+1}$
	only appear in the term 
    $(y_{i_1^{(t+1)}\,t+1} \cdots y_{i_{k-1}^{(t+1)}\,t+1}) f_t$, 
    we conclude that $f_{t+1}$ is a non-zero $(t+1)$ minor of $A$. 
    $\triangleright$ 

    We can now prove the following claim which immediately 
    implies the assertion of the lemma.

    \noindent {\sf Claim:} $f_t \not\in L_H^\K(d)$. 

    \noindent $\triangleleft$ Every polynomial $f$ in the ideal 
    $L_H^{\K}(d)$ is an $S$-linear combination of generators of the ideal. 
    Therefore, every monomial in the support of $f$ is a multiple of a 
    monomial of the form $y_{i_{1} j} y_{i_{2} j}\ldots y_{i_{k} j}$. 
    But monomials in the support of minors of $A$ are squarefree and are 
    divisible by only $k-1$ variables of the same second index.
    Thus $f_t \not\in L_H^{\K}(d)$ and the assertion follows.
    $\triangleright$
  \end{proof}

  The following definition introduces the key combinatorial structure for 
  our needs.

  \begin{defn}
	  For integers $k,c > 0$ such that $0 < k -1 \leq n-c$ let 
    $W$ be a set of $(k-1)$-subsets of $[n-c]$ with $|W| > 0$.
	  .
    By $H_{W,c}$ we denote the 
    hypergraph $$\Big([n], \Big\{ \{i_1,\ldots,i_{k}\}~\Big|~ 
    \{i_1,\ldots,i_{k-1}\} \in W, ~ i_k\in \{n-c+1,\ldots,n\}\Big\}\Big).$$
  \end{defn}

  \begin{prop}\label{p20}
    Let $H=([n],E)$ be a $k$-uniform hypergraph.
	  For integers $k,c > 0$ such that $0 < k -1 \leq n-c$ let 
    $W$ be a non-empty set of $(k-1)$-subsets of $[n-c]$.
    If $L_H^{\K}(d)$ is prime, then $H$ does not contain $H_{W,c}$ for any
    set $W$ of $(k-1)$-subsets of $[n-c]$ with $|W|+c > d$.
  \end{prop}
  \begin{proof}
    Suppose $H$ contains $H_{W,c}$ for some $W$ and $c$ such that 
    $|W|+c > d$. Set $w = |W|$. We can assume that $w +c=d+1$. Set 
    $T=S/L_H^{\K}(d)$, 
    $A=(y_{i_{1}j} \cdots y_{i_{k-1}j})_{\{i_1,\ldots, i_{k-1}\} \in W, j \in [d]}$ 
	  and $B=(y_{jr})_{(r,j) \in \{n-c+1,\ldots,n\} \times [d]}$.
	  By our assumptions $A$ and $B$ are not empty matrices.
    The entries of $AB$ are the generators of the ideal 
    $L_{H_{W,c}}^{\K}(d)$, hence $AB=0$ in $T$. Since $T$ is a domain, 
    the matrix equation $AB=0$ can be seen as an identity over the field of 
    fractions of $T$. By linear algebra we deduce that 
    $\rank(A)+\rank(B) \leq d$. From 
    $w+c=d+1$ it follows that either $\rank(A) < w$ or $\rank(B) < c$. 
    Therefore, $I_w(A)=0$ or $I_c(B)=0$. But from \ref{lem:nonvan} it follows 
    that $I_w(A)$ and $I_c(B)$ are non-zero in $T$. This is a contradiction. 
    Hence $H$ does not contain any $H_{W,c}$ for $W$ and $c$ such that
    $w+c > d$. 
  \end{proof}

  The following remark will serve as the induction base for the proof of
  \ref{th1} and \ref{th2}.

  \begin{rem}
    \label{rem:base} 
    For $n \leq k$ either $L_H^\K(d) = (0)$
    or $L_H^\K(d) = (f_{1,\ldots, n}^d)$. For $d \geq 2$ the polynomial
    $f^d_{1,\ldots,n}$ is irreducible. Hence for $n \leq k$ and $d \geq 2$,  
    $L_H^\K(d)$ is a complete intersection and prime.
  \end{rem}

  \begin{proof}[Proof of \ref{th1}]
    We prove the theorem by induction on $n$. The induction base 
    for $n \leq k$ and all $d\geq 2$ is laid in \ref{rem:base}.

    Assume that $n > k$. 
    Let 
    $A=(y_{i_1j} \cdots y_{i_{k-1}j})_{\{i_1,\ldots, i_{k-1}\} \in U,j \in [d]} \in R^{u \times d}$.
    Then $A^T$ defines a linear map $R^{u} \longrightarrow R^{d}$. 
    From $L_{H'}^{\K}(d)= L_{H}^{\K}(d) \cap S' $ it follows that 
    $L_{H'}^{\K}(d)$ is prime and hence $R$ is a domain. By induction we 
    deduce that $R = S'/L_{H'}^{\K}(d)$ is a complete intersection. 
    Thus we can apply \ref{Sym}.

    The definition of $U$ implies that $H_{U,1} \subseteq H$. 
    Since $L_{H}^{\K}(d)$ is prime, \ref{p20} implies $u+1 < d+1$, i.e. $ u < d $. 
    By \ref{lem:nonvan} we have that the minors of the matrix $A$ are non-zero in $R$. 
    Therefore, $I_u(A) \not=0$ in the domain $R = S'/L_{H'}^{\K}(d)$.

    Hence by \ref{Sym}(2) it follows that 
    $\height(I_t(A)) \geq u - t + 2 > u-t+1$ for all $t = 1,\ldots, u$.
    Then by \ref{Sym}(1) $S/L_H^\K(d)$ is a complete intersection.
  \end{proof}

In order to apply \ref{Sym}(2) to prove \ref{th2} 
we need the following crucial lemma about heights of ideal of minors.

\begin{lem}\label{height of minors}
  For integers $ n, u > 0$ and $ k \geq 2$ assume $U = \{ A_1,\ldots, A_u\}$ for distinct 
  $(k-1)$-subsets $A_1,\ldots, A_u$ of $[n]$.
	For elements $(y_{ij})_{i \in [n],j\in[d]}$ from the Noetherian ring $R$ and 
	variables $(y_{i\,d+1})_{i \in [n]}$ define 
	$m_{ij} = \prod_{\ell \in A_i} y_{\ell j}$ for $i \in [u]$ and $j \in [d+1]$.
  Consider the matrix
  $$
	M= \left[
	\begin{array}{cccc}
	m_{11} & m_{12} & \cdots & m_{1d} \\
	m_{21} & m_{22} & \cdots & m_{2d} \\
	\vdots & \vdots & & \vdots \\
	m_{u1} & m_{u2} & \cdots & m_{ud}
	\end{array}
	\right]
  $$
  with entries in $R$ and the matrix 
  $M'$ arising from $M$ by adding the new column $$ [m_{1d+1},\cdots,m_{ud+1}] ^T$$ with entries in 
	$T =R[Y] = R[y_{1d+1},\cdots,y_{nd+1}]$.  

  Then for all $ 1 < t \leq u $ we have 
	\begin{align}
		\label{eq:height}
		\height \, I_t(M') & \geq \min \{\height \,I_{t-1}(M) , \height\, I_t(M)+1\} .
	\end{align}
\end{lem}
\begin{proof}
  Let $ P $ be a minimal prime ideal of $ I_t(M') $ in $ T $. If $ P \supseteq I_{t-1}(M) $ then $ \height P \geq \height I_{t-1}(M)$.
  If $ P \nsupseteq I_{t-1}(M)  $ then there exists a $ (t-1) $ minor $ f $ of $ M $ which is not in $ P $. 
  We assume that $ f $ is the minor corresponding to the first $ (t-1) $ rows and columns of $ M $. 
  If follows that $ \height \, P=\height\, PR_f[Y]$. Furthermore $ PR_f[Y]  $ contains $ I_t(M) R_f[Y]$ and 
  elements $ (m_{id+1} -f^{-1}g_i : i=t,\cdots,u ) $ with $ g_i \in T $. Since every single element 
  $ m_{id+1} -f^{-1}g_i $ is algebraically independent over $ R_f $, we have 
  $$ \height\, PR_f[Y] \geq \height \,I_t(M)R_f[Y] +1 \geq \height \,I_t(M)+1.$$

  Now \eqref{eq:height} follows.
\end{proof}

Now we are in position to provide the proof of \ref{th2}. Our arguments are similar to those used in 
\begin{proof}[Proof of \ref{th2}]
  We prove the theorem by induction on $n$ and $d$. The induction base 
  for $n \leq k$ and $d\geq 2$ is laid in \ref{rem:base}.

  Let $n > k$. 
  Then by \ref{rem:symmetric} $S/L_{H}^{\K}(d)$ is the symmetric 
  algebra of the cokernel of the $R$-linear map 
    $R^{u}\xrightarrow{A^T} R^{d}$ given by the $u \times d$-matrix 
    $A=(y_{i_1\,j} \cdots y_{i_{k-1}\,j})_{\{i_1,\ldots,i_{k-1}\} \in U,j\in [d]}$. 
    Since $L_H^{\K}(d-1)$ is a complete intersection, it follows that 
    $L_{H'}^{\K}(d-1)$ is a complete intersection. Therefore, by induction 
    on $n$ we have that $L_{H'}^{\K}(d)$ is prime 
    and hence $R$ is a domain. Since the polynomials 
    $f_{\{i_1,\ldots,i_{k-1},n\}}^d$ with 
    $\{i_1,\ldots,i_{k-1}\}\in U$ are a regular sequence contained in the ideal 
    $(y_{nl} :~ 1\leq l \leq d) $ we have that $u \leq d$ and by \ref{lem:nonvan} 
    $I_u(A)\neq 0$ in $R$. Therefore, by \ref{Sym}(2) we have that 
    $L_H^{\K}(d)$ is prime if and only if 
    $\height(I_t(A)) \geq u-t+2$ in $R$ 
    for every $1 \leq t \leq u$. Equivalently we should prove that 
    \begin{align*}
      \height(I_t(A)+L_{H'}^{\K}(d)) \geq u-t+2+h
    \end{align*}
    in $S'$ for every $1 \leq t\leq u$, where $h=\height(L_{H'}^{\K}(d))= |E| - u$. 
    Consider the weight vector $\w \in \RR^{n \times (d)}$ that is defined by $\w_{ij}=1 $ and $\w_{id} =0$ 
    for all $j\in [d-1] $ and $i \in [n] $. By construction the initial forms of the standard generators of 
    $L_{H'}^{\K}(d)$ are the standard generators of  $L_{H'}^{\K}(d-1)$. 
    In addition, the standard generators of $I_t(A)$ coincide with their initial forms with respect to 
    $\w$. Therefore
    \begin{align*}
      \iin_\w(I_t(A)+L_{H'}^{\K}(d))\supseteq I_t(A)+ L_{H'}^{\K}(d-1)
    \end{align*}
    and it is enough to prove that $\height(I_t(A)) \geq u-t+2 $ in $ R' $ for every $ 1 \leq t \leq u $, 
    where $ R'=S'/L_{H'}^{\K}(d-1)$. By construction, we can write $R'=R''[y_{1d},\ldots,y_{(n-1)d}]$ 
    with $ R''=\K[y_{ij} :~(i,j) \in [n-1]\times [d-1]]/L_{H'}^{\K}(d-1)$. Let $A'$ be the 
    matrix $ A $ with the $d$\textsuperscript{th} column removed. Then $S/L_H^{\K}(d-1)$ can be seen as the symmetric 
    algebra of the $R''$-module presented as the cokernel of the map $(R'')^u \xrightarrow{{A'}^T} (R'')^{d-1}$. By assumption  
    $ S/L_H^{\K}(d-1) $ is a complete intersection. Hence by \ref{Sym}(1) we have 
    $\height(I_t(A')) \geq u-t+1 $ in $ R'' $ for every $ 1 \leq t \leq u $. 
    

    Since $A$ is obtained from $A'$ by 
    adding a column of square free monomials of degree $ k-1 $ over $ R'' $ by \ref{height of minors} we have 
    \begin{align*}
    \height(I_t(A))\geq \min \{ \height(I_{t-1}(A')), \height (I_t(A'))+1\} \geq u-t+2.
    \end{align*}  
  \end{proof}

  \begin{cor}
    Let $H=([n],E)$ be a $k$-uniform hypergraph. If $L_H^{\K}(d)$ 
    is prime (complete intersection), then $L_H^{\K}(d+1)$ is prime (complete intersection).
  \end{cor}

\section{Positive matching decomposition for hypergraphs}
\label{sec:3}

In this section we use a combinatorial approach to provide an upper 
bound for the least $d$ for which the ideal $L_H^{\K}(d)$ is prime when 
$ H $ is a $3$-uniform tree.
As mentioned in the introduction our definition of a $k$-uniform tree is more 
restrictive than other definitions from the literature
(see \cite{b}, \cite{bdcv}, \cite{np}).

\begin{defn}
  Let $k \geq 2$. A $k$-uniform hypergraph $ H=(V,E) $ is called a ($k$-uniform) \textit{tree} if 
  \begin{itemize}
    \item[(T1)] For each pair of edges $ e,e' \in E $ we have 
	    $ |e \cap e'| \leq 1 $.
    \item[(T2)]
       For each pair of vertices $ v,v' \in V $, for which there is no 
		  edge in $E$ containing both, there exists a unique 
       sequence $ e_1,\ldots , e_r \in E $ such  that:\\
        \begin{itemize} 
          \item[(a)] $ v \in e_1 $ and $ v' \in e_r $ and $v,v' \not\in 
		   e_2,\ldots,e_{r-1} $,
	  \item[(b)] For each $ 1 \leq i \leq r-1 $, we have $| e_i \cap e_{i+1} |=1$,
	  \item[(c)] For each $ 1 \leq i\neq j \leq r $ where $|i-j| \geq 2$ we have $ e_i \cap e_j = \varnothing$.
        \end{itemize} 
  \end{itemize} 
   We call the unique sequence of edges from (T2) for $ v $ to $ v' $ the 
	$vv'$-sequence and if 
	$ v \in e $ 
	and 
	$ v' \in e' $
	 we call the the unique sequence from (T2) also 
	$ee'$-sequence.
\end{defn}


Clearly, for $k=2$ the $k$-uniform trees are the trees in the usual graph 
theoretic sense.

For the rest of this section it will be convenient to use the following notions.
\begin{itemize}
\item
 For a hypergraph $ H=(V,E) $  and $v \in V$ we denote by 
	$\deg_H(v)=|\{e \in E |~ v\in e \}| $ the \textit{degree} of $v$
	and  $ \Delta(H)= \max \{\deg_H(v) |  ~v \in V \}$ denotes the maximal degree of $ H $.
\item
 A hypergraph $ H=(V,E) $ with $ n $ vertices and $ m $ edges is called star when $ n-1 $ vertices have degree $ 1 $ and a single vertex has degree $m$.

\end{itemize}

Note that $k$-uniform hypergraphs which are stars are $k$-uniform trees.
For that conclusion condition (T2)(c) is crucial.
The following lemma plays an important role in the proof of \ref{tree} 
and \ref{TREE}.

\begin{lem}\label{degree1 vertex}
  Let $k \geq 2$ and $H=(V,E) $ be a $k$-uniform hypergraph which is a tree. 
  Then there exists at least one vertex 
  $ v \in V $, with $\deg_{H}(v)=1 $.
\end{lem}
\begin{proof}
  If the hypergraph $ H $ is a star, then the assertion is obvious. 
  Assume $ H $ be a hypergraph which is not a star. Then for each vertex 
  $ v' \in V $ there exists at least one vertex $ v \in V $, such that 
  $ v' $ and $ v $ do not lie in the same edge. For each such vertex 
  $ v \in V $, by (T1) and (T2), there exists a unique 
  $v'v$-sequence $e_1,\ldots,e_r \in E$. 

  Now let $ v' \in V $ be a fixed vertex of $ H $ and
  $v \in V$ be a vertex for which the unique $v'v$-sequence has maximal length
  $ r $. We claim that $\deg_H(v)=1 $. 
  Assume that $\deg_H(v)\geq 2 $. Then there is at least one edge 
  $ e' \neq e_r$ which contains the vertex $ v $. 
	Let $ v''$ be a vertex in $e' \setminus \{v\}$. By definition there is a unique 
	$v'v''$-sequence $\bar{e}_1,\ldots,\bar{e}_s \in H $. 
  Maximality of $r$ shows that $s \leq r$. 
  The sequence $ e_1,\ldots,e_r,e'$ is also from $ v' $ to $ v'' $ and of 
	length $r+1 > s$ and satisfies (T2)(a) and (b).
  Thus by uniqueness $e' \cap e_i \neq \varnothing$ for some $i < r$
	which implies by (T1) $|e' \cap e_i| =1$. Let $i$ be minimal
	with that property. Then $e_1,\ldots, e_i ,e'$ is a
	$v'v$-sequence. This fact contradicts the uniqueness of the $v'v$-sequence. 
	and $ \deg_H(v)=1 $ follows.
\end{proof}

In the following we write $V(M)$ for the vertex set of a subset $M \subseteq E$  of the
set of edges of a hypergraph $H = (V,E)$. 

\begin{prop}\label{tree}
  Let $k \geq 2$ and let $ H=(V,E) $ be a $k$-uniform hypergraph which is a tree with 
	$ \Delta(H) \geq 2 $. Then there is a positive matching $M \subseteq E$
	such that $V(M)$ contains all vertices $v$ of degree $\deg(v) \geq 2$.
\end{prop}
\begin{proof}
  For $k=2$ this is a consequence of \cite[Lemma 5.4.(3)]{ac-vw}. Thus
  we may assume $k \geq 3$. 

  Let $ H $ be a $k$-uniform tree with $ \Delta(H)\geq 2$. 
  We define sequences 
  $ (M_i)_{i\geq 1} $, $ (M'_i )_{i\geq 1}$ and $ (\overline{M}_i)_{i\geq 1} $ 
  of subsets of $ E $ as follows:

  By \ref{degree1 vertex} there is $ v \in V $ with $ \deg_H(v)=1 $. 
	Let $\bar{e} \in E$ be the unique edge 
  containing $v$. Set $M'_0=\varnothing$, 
	$ M_1:=\{\bar{e} \} $,
	$\overline{M}_1= \{\bar{e}\}$ and
	$M'_1:=\{ e' \in E\setminus (M_1 \cup M'_0) \mid e'\cap \bar{e} \neq \varnothing\} $.
	Since $\Delta(H) \geq 2$ there are other edges than $\bar{e}$ in $H$. 
	From the fact that $H$ is a tree it then follows that 
	$V(M_1') \setminus V(M_1) 
	 \neq \varnothing$. 

	Assume that for some $i \geq 1$ we have constructed 
	$M_{i}$ and $M_{i}'$ and 
	\begin{itemize}
		\item[(M1)] $V(M_{i}') \setminus V(M_{i}) \neq \varnothing$.
		\item[(M2)]  $M'_i$ is the set of all edges not contained in 
			$M_i$ 
			intersecting some $e \in M_i$ non-trivially
	\end{itemize}

	Define 
	\begin{align} 
	   \label{eq:ovm} 
           \overline{M}_{i+1}  :=\{ e \in E\setminus 
	(M_{i} \cup M'_{i}) \,|~ \exists e' \in M'_{i}
	 : e \cap e' \neq \varnothing\, \}.
	\end{align} 

   \medskip

	\noindent {\sf Claim(a):} $ \overline{M}_{i+1} $ is empty or a disjoint union of stars.

	Assume $\overline{M_{i+1}}$ is non-empty. Let $e \in \overline{M}_{i+1}$ then by definition 
	$e \cap e' = \{v'\}$ for
	some $e' \in M_{i}'$. If $v' \in V(M_i)$ then $e$ also intersects
	an edge from $M_i$ in $v'$ and hence by (M2) $e \in M_i'$ contradicting the
	choice of $e$. 

  We set 
	$ V':=V(M'_{i}) \setminus V(M_{i}) \neq \varnothing$. 
	It follows from the definition that every 
	$e \in \overline{M}_{i+1}$ contains some 
	$v' \in V'$. 

        For $ v' \in V' $ define 
	$ S_{v'} :=\{ e \in \overline{M}_{i+1} \mid v' \in e \} $. 
   By (T1) for every pair $ e', e'' \in S_{v'} $ we have 
   $ e'\cap e''=\{v'\} $.
   It follows that 
	$S_{v'} $ is a 
   star for each $ v' \in V'$. 

	By definition for $ e \in \overline{M}_{i+1} $
	there exists $ e' \in M_{i}'$ such that 
   $e \cap e'=\{v'\}\ $ for some $v' \in V'$. Hence 
   $ e \in S_{v'}$. This implies that 
	$\overline{M}_{i+1} = \bigcup_{v'\in V'} S_{v'}$. Let 
   $ e' \in S_{v'} $ and $ e''\in S_{v''}$ for some $ v' \neq v''$ in $ V'$. 
   Assume $e' \cap e''=\{w\} $. Then by prolonging the unique $vv'$-sequence 
   by $e'$ and the unique $vv''$-sequence by $e''$ we construct two 
   different $ vw $-sequences.
	This contradicts (T2) from the definition of a tree. 
	Therefore, $ e'\cap e''=\varnothing $ and $ \overline{M}_{i+1} $ 
	is a disjoint union of stars.

  \medskip


  Put $$M_{i+1}=M_{i}\cup \{\text{ choose one edge from each star in } \overline{M}_{i+1}\} $$ and 
	$$M'_{i+1}= M'_{i}\cup (\overline{M_{i+1}} \setminus M_{i+1}) \cup \{ e \in E \setminus (M_{i+1} \cup M_{i}') |~\exists e' \in M_{i+1}, ~e \cap e' \neq \varnothing \} .$$ 

By (M2) for $M_i$ and $M_i'$ and 
the construction of $M'_{i+1}$ and 
$\overline{M}_{i+1}$ it follows that
$M'_{i+1}$ is the set of all edges outside of 
$M_{i+1}$ intersecting some $e \in M_{i+1}$
non-trivially.  Thus (M2) is satisfied.

Since there are only finitely many edges and by the construction
of $M_{i+1}$ and $M'_{i+1}$ it follows from (T2) that 
there is an $i$ such that $M_{i+1} \cup M'_{i+1} = E$.
We then stop our construction and set 
$ M:=M_{i+1} $ and $M':= M'_{i+1}$. 
Note that from (T1) and (T2) it follows that if 
$V(M_{i+1}') \setminus V(M_{i+1}) = \varnothing$ then 
$M_{i+1}' \cup M_{i+1} = E$.
If $M_{i+1} \cup M'_{i+1} \neq E$ and 
$V(M_{i+1}') \setminus V(M_{i+1}) \neq \varnothing$ then 
(M1) is satisfied then we increase $i$ by one and continue in
\eqref{eq:ovm}.

\medskip

\noindent {\sf Claim(b):} $ M $ is a matching.

  Let $e,e' \in M$. Let $i$ and $j$ be such that
  $ e \in M_{i+1}\setminus M_{i}$ and
  $ e' \in M_{j+1}\setminus M_{j}$ 
  for some 

  We have to consider two cases.

  In the first case assume $i=j$. 
  Then $ e $ and $ e' $ are chosen from disjoint 
  stars in $\overline{M}_{i+1}$ and thus $ e \cap e'=\varnothing $ .

  In the case $i \neq j$ we may assume $j < i$. 
  Since $M'_{j}$ contains all edges outside of $M_j$ whose intersection with 
  some edge in $M_j$ is non-trivial, it follows $e' \in M'_j$.
  But the edges from $M_i$ are chosen outside of $M'_{i-1} \supset M'_j$.
  This yields a contradiction and it follows that $e \cap e' = \emptyset$.


  \noindent {\sf Claim(c):} $ E=M\cup M' $ and $ M\cap M'= \varnothing$.

  We already have seen that $ E=M\cup M' $. It remains to show that 
  $ M\cap M'= \varnothing$. 
  Assume there is $e' \in M\cap M'  $. Then there exist smallest integers 
  $ i $ and $ j $ such that $ e' \in M_i $ and $ e' \in M'_j $. 
  By definition the edges from $M_i$ are not contained in $M'_{i-1}$ and 
  hence in any $M'_\ell$ for $\ell < i$. 
  Again definition the edges from $M'_j$ are not contained in $M_{j}$ and hence 
  in any $M_\ell$ for $\ell \leq j$. 
  It follows that $M_i \cap M'_j = \emptyset$ contradicting the assumption. 
  The assertion then follows.


\medskip

 \noindent {\sf Claim(d):} $\{~v'\in V~|~\deg_H(v') \geq 2\} \subset V(M)$.

 \medskip

Consider $ A=\{~v' \in V \lvert~\deg_H(v) \geq 2,~v'\notin V(M)\} $. To prove 
the claim we have to show $ A=\varnothing $. 

Assume that $A \neq \varnothing$. 
Let $ v'\in A $ then by definition we have $ \deg_H(v')=s \geq 2$.
Let $v \in V$ be the chosen vertex of degree $1$ in $H$ and $\bar{e} \in E$
the unique edge containing $v$. 
Let further
$ \bar{e}=e_1,\ldots,e_r $ be a $ vv' $-sequence. We choose 
$ v'\in A $ such that $ r $ is minimal. From $v' \in A$ we know that 
$e_r \not\in M$. Let $ e_r \cap e_{r-1}=\{v''\} $.
It follows that $ \deg_H(v'')\geq 2 $. The length of a sequence from $ v $ to 
$ v'' $ is smaller than $ r $ and therefore $v'' \not\in A$. 
Hence one of the edges that contains $ v'' $ is in some $ M_{i} \subseteq M $. 
Let $i$ be minimal with that property. 
By $e_r \not\in M$ we know $e_r \not\in M_i$. 
Thus by (M2) and its definition $ e_r $ is in $ M'_i $
and the the other $\deg_H(v')-1$ edges containing $v' $ 
form a star from $ \overline{M}_{i+1} $. 
By definition one of the edges of this star is in 
$ M_{i+1} $. It follows that $v' \not\in A$.  This is a contradiction. 
Therefore, $ A $ is empty and the claim follows.

\medskip

 \noindent {\sf Claim(e):} $M$ is positive

\medskip

For proving the positivity of $ M $, we consider the following weight on vertices.
For $v \in V$ let $i$ be the minimal index such that $v \in V(M_i)$ or $v \in V(M'_i)$. Then set

  \begin{equation*}
\w(v) := \left\{
\begin{array}{rl}
	1 & v \in V(M_i) \\
	-1 &  v \not\in V(M_i) \text{ and } v \in V(M'_i)
\end{array} \right.
\end{equation*}
Since by (c) for any $v \in V$ there is an $i$ such that $v \in M_i$ or
$v \in M'_i$ the weight is well defined. 
It remains to be shown that $\w(e) > 0$ for $e \in M$ and $\w(e) < 0$ if $e \in M'$.

If $e \in M$ then there is an $i$ such that 
$e \in M_i \setminus M_{i-1}$. Then there is 
at most on vertex $v' \in e$ such that $v' \in M'_{i-1}$ the other vertices 
are contained in $V(M_i) \setminus (V(M_{i-1}) \cup V(M'_{i-1}))$. 
If $i > 1$ then there is a unique vertex $v'$ with this property and if $i=1$ there is none. 
It follows that $\geq k-1 \geq 2$ of the vertices have weights $1$ and 
$\leq 1$ have
weight $-1$. Thus $\w(e) > 0$. 

If $e \in M'$ then there is an $i$ such that 
$e \in M'_i \setminus M'_{i-1}$. 
Then $e$ either contains a 
unique vertex $v'$ which is contained in $V(M_i)$ or a unique vertex 
contained in $V(M'_{i-1})$.  
All other vertices lies in $V(M'_{i}) \setminus (V(M_i) \cup V(M'_{i-1}))$.
It follows that $\geq k-1 \geq 2$ of the vertices have weights $-1$ and $\leq 1$ have
weight $1$. Thus $\w(e) < 0$. 
\end{proof}

\ref{tree} allows us to prove \ref{TREE} by induction of $\Delta(H)$.

\begin{proof}[Proof of \ref{TREE}]
  Clearly, $ \pmd(H)\geq \Delta(H) $. We use induction on $ \Delta(H) $ to 
  show that $ \pmd(H)=\Delta(H) $. For $ \Delta(H)=1 $ there is nothing to 
  prove. Assume $ \Delta(H)\geq 2 $.

  By \ref{tree} there is a positive matching $M$ such that 
  $V(M)$ contains all vertices of degree $\geq 2$. 
  Let $ H' $ be the hypergraph obtained  from $ H $ by removing $ M $. 
	The conditions (T1) and (T2) for $H$ are inherited by the 
	connected components of $ H' $. Thus each connected component of
	$H'$ is a tree. The fact that 
  $V(M)$ contains all vertices of degree $\geq 2$ implies that
	for a connected component $H''$ of $H'$ we have $\Delta(H'') \leq
	\Delta(H)-1$. 
	By induction there we have $\pmd(H'') \leq \Delta(H'') \leq 
	\Delta(H) -1$. Since the pmd of a hypergraph is easily seen to be 
	the maximum of the pmd's of its connected components it follows
	that $\pmd(H') \leq \Delta(H)-1$. Hence $\pmd(H) \leq \Delta(H)$.
\end{proof}

Next we are interested to find an upper bound for $ \pmd(H) $ as a function
of the number of vertices $n$ when $ H $ is a 
$3$-uniform hypergraph. For $2$-uniform hypergraphs, that is
graphs, such a linear bound has been provided in \cite[Lemma 5.4]{ac-vw}. 
Below we provide an attempt to apply a similar method in order to show that 
$ \pmd(H) $ is bounded above by a quadratic function in $n$ when $ H $ is a 
$3$-uniform hypergraph. For that we extend the 
strategy of \cite[Lemma 5.4]{ac-vw} to $3$-uniform hypergraphs.  
We can prove the following.

\begin{prop}\label{matching}
  Let $ H=(V,E) $ be the complete $ 3 $-uniform hypergraph on $ n $ vertices
	and with $\binom{n}{3}$ edges.
  Then for 
  every $ 3 \leq l_1 \leq 2n-3 $ and $ 5 \leq l_2 \leq 2n-1 $, the set 
  $ E_{l_1,l_2}=\{\{a,b,c\}\in E \mid a < b < c,~~ a+b=l_1, b+c= l_2\}$ is a 
  matching and $ E=\bigcup_{ l_1,l_2}E_{l_1,l_2} $. 

  In addition, the cardinality of the set 
	$$E_n:=\{(l_1,l_2)\mid \text{ there exist } 1\leq a < b < c\leq n,~~ 
	l_1=a+b, l_2=b+c\}$$ is $ \frac{3}{2} n^2-\frac{15}{2}n+10 $.
\end{prop}
\begin{proof}
  Let $ e=\{a < b <c \} , e'=\{a'<b'< c'\} \in E_{l_1,l_2}$ for some 
  $ 3 \leq l_1 \leq 2n-3 $ and $ 5 \leq l_2 \leq 2n-1 $. Assume $e \neq e'$ and
  $ e\cap e'\neq \varnothing $. 

  If $a=b'$, $b=a'$, $b=c'$ or $c=b'$ then the conditions $a+b=a'+b' = l_1$
  and $b+c=b'+c'=l_2$ yield a contradiction to the order of the elements
  in $e$ and $e'$. 
  Assume $a = c'$ Then $a' < b' <c'=a <b$ which again contradicts $a'+b'=a+b$. 
  Hence $a=a'$ or $b=b'$ or $c=c'$. We consider the case
  $ a=a' $. The other cases follow analogously. Since $ a+b=l_1=a'+b' $, we 
  have $ b=b' $ and by 
  $ b+c=l_2=b'+c' $ we have $ c=c' $. Thus $e=e'$ contradicting the 
  assumptions. It follows that $E_{l_1,l_2}$ is a matching.

  The fact that 
  $ E=\bigcup_{ l_1,l_2}E_{l_1,l_2} $ is obvious. 

  For proving that 
  $| E_n| = \frac{3}{2} n^2-\frac{15}{2}n+10$ we use induction on $n$.

  For $ n=3 $ we have $|E_n| =|\{(3,5)\}| = 1 = \frac{3}{2}\,3^2- \frac{15}{2}\,3 +10$. 
  Suppose we have proved for some $n$ that 
  $|E_n|= \frac{3}{2} n^2-\frac{15}{2}n+10 $. 
  In the induction step we show that 
  $|E_{n+1} \setminus E_n|= 3n-6$ from which the assertion follows
	by an elementary summation.

  Let $(l_1,l_2)\in E_{n+1} \smallsetminus E_n$. Then there exist 
  $ 1\leq a< b< c\leq n+1 $ such that $ l_1=a+b $ and $ l_2=b+c $. By
  $ (l_1,l_2)\not\in E_n $ we have $ c=n+1 $. 

  We show that either $ a=1 $ or $ b = n,n-1 $. 
  In case $ a \geq 2 $ and $ b \leq n-2 $ we have $ 1\leq a-1 < b+1< n$. 
  This shows that $ (l_1,l_2) \in E_n $ for $(a-1,b+1,n)$. 

	Consider the case $b = n-1$ (resp., $b=n$) 
	then any $1 \leq a \leq n-2$ (resp., $1 \leq a \leq n-1$)
	determines a different $(l_1=a+n-1,l_2=2n)$ (resp., 
	$(l_1=a+n,l_2 = 2n+1)$) not contained in $E_n$ yielding 
  $n-2+n-1 = 2n-3$ cases.

  Now assume $a=1$ and $1 < b \leq n-2$. Assume we have 
	$1 \leq a' < b' < c' \leq n$
  with $a+b = a'+b'$ and $b+c = b'+c'$.
  From $c = n+1$ we infer $c'< c$ and $b' > b$. But $a' \geq a$ and
  hence $a'+b' > a+b$ contradicting $a+b=a'+b'$. 
  It follows that each pair $a=1 < b \leq n-2$ determines $(l_1,l_2)
  \not\in E_n$. 
  By $l_2 = b+c \leq n-2+n+1 = n-1 < a+b \leq 1+n-2 = n-1$.  
  Since there are $n-3$ cases $1=a < b \leq n-2$ this count contributes
  $n-3$ cases.

  Thus there are $2n-3+n-3=3n-6$ elements in $E_{n+1} \setminus E_n$
	as desired. 
\end{proof}

In order to complete our the strategy following \cite[Lemma 5.4]{ac-vw}
we would have to prove the following conjecture.

\begin{conj}\label{CONJ}
 Let $ H=(V,E) $ be a $ 3 $-uniform hypergraph as in \ref{matching}. Then $ E_{l_1,l_2}=\{\{a,b,c\}\in E \mid a < b < c,~~ a+b=l_1, b+c= l_2\}$ is a positive matching.
\end{conj}

If the preceding conjecture is true then the following conjecture is implied.

\begin{conj}
  \label{con:conj}
Let $ H=(V,E) $ be a $ 3 $-uniform hypergraph with $ n $ vertices. Then $ \pmd(H) \leq \frac{3}{2} n^2-\frac{15}{2}n+10 $.
\end{conj}\label{conj}

As mentioned in Section \ref{section0} by \cite{LANS} the upper bound for the 
symmetric tensor rank in case $ k=3 $ is a cubic function in $ n $. 
By \ref{con:conj} the $ \pmd(H) $ for a $3$-uniform hypergraph $H$ is bounded 
from above by a quadratic function in $ n $. 
Thus if the conjecture holds, for 
$3\binom{n}{2}-6n+10 \leq d \leq \binom{n+2}{3}-n+1 $ every coordinate section 
of $ S_{n,k}^d $ is irreducible. By \cite[Theorem 5.4(1)]{ac-vw} 
there is a linear upper bound of $ \pmd(H) $ for a $ 2 $-uniform hypergraph,
in \ref{conj} the $ \pmd(H) $ for a $ 3 $-uniform hypergraph $ H $ is bounded 
from above by a quadratic function in $ n $. Therefore, one can speculate that 
in general for $ k $-uniform hypergraphs $H$ the value of $\pmd(H)$ is bounded
from above by a polynomial of degree $ k-1 $ in number $n$ of vertices.

\end{document}